\documentclass[a4paper,11pt]{amsart}
\usepackage{amssymb}
\usepackage[T1]{fontenc}

\theoremstyle{plain}
\newtheorem{theorem}{Theorem}

\theoremstyle{definition}

\theoremstyle{remark}

\numberwithin{equation}{section}

\newcommand{\ko}{{\mathcal O}}

\newcommand{\IC}{{\mathbb C}}

\newcommand{\IP}{{\mathbb P}}

\newcommand{\IZ}{{\mathbb Z}}

\newcommand{\gothh}{{\mathfrak h}}

\DeclareMathOperator{\orb}{orb}

\newcommand{\isom}{\cong}
\newcommand{\tensor}{\otimes}

\DeclareMathOperator{\Sym}{Sym}
\DeclareMathOperator{\Hom}{Hom}

\DeclareMathOperator{\Ext}{Ext}

\DeclareMathOperator{\LieSp}{Sp}

\DeclareMathOperator{\LieSl}{SL}

\DeclareMathOperator{\LieGl}{GL}

\DeclareMathOperator{\LieSO}{SO}
\DeclareMathOperator{\LieSU}{SU}

\DeclareMathOperator{\Spec}{Spec}
\DeclareMathOperator{\Hilb}{Hilb}

\DeclareMathOperator{\sing}{sing}
\DeclareMathOperator{\red}{red}

\newcommand{\q}{q}

\begin{document}

\title{A symplectic resolution for the binary tetrahedral group}
\date{6 October 2008}

\author{Manfred Lehn \& Christoph Sorger}
\address{
Manfred Lehn\\
Fach\-be\-reich Ma\-the\-ma\-tik und In\-for\-ma\-tik\\
Johannes Gu\-ten\-berg-Uni\-ver\-si\-t\"{a}t Mainz\\
D-55099 Mainz } \email{lehn@mathematik.uni-mainz.de}

\address{
Christoph Sorger\\
Laboratoire de Math\'{e}matiques Jean Leray (UMR 6629 du CNRS)\\
Universit\'{e} de Nantes\\
2, Rue de la Houssini\`{e}re\\
BP 92208\\
F-44322 Nantes Cedex 03, France}
\email{christoph.sorger@univ-nantes.fr}

\subjclass{Primary 14B05; Secondary 14E15, 13P10}

\begin{abstract}
We describe an explicit symplectic resolution for the quotient
singularity arising from the four-dimensional symplectic represenation
of the binary tetrahedral group.
\end{abstract}

\maketitle


Let $G$ be a finite group with a complex symplectic representation $V$. The
symplectic form $\sigma$ on $V$ descends to a symplectic form $\bar \sigma$
on the open regular part of $V/G$. A proper morphism $f:Y\to V/G$ is a symplectic
resolution if $Y$ is smooth and if $f^*\bar\sigma$ extends to a symplectic form
on $Y$. It turns out that symplectic resolutions of quotient singularities are
a rare phenomenon. By a theorem of Verbitsky \cite{Verbitsky}, a necessary condition
for the existence of a symplectic resolution is that $G$ be generated by
symplectic reflections, i.e.\ by elements whose fix locus on $V$ is a linear
subspace of codimension 2. Given an arbitrary complex representation $V_0$
of a finite group $G$, we obtain a symplectic representation on $V_0\oplus V_0^*$,
where $V_0^*$ denotes the contragradient representation of $V_0$. In this case,
Verbitsky's theorem specialises to an earlier theorem of Kaledin \cite{Kaledin}:
For $V_0\oplus V_0^*/G$ to admit a symplectic resolution, the action of $G$ on $V_0$
should be generated by complex reflections, in other words, $V_0/G$ should be
smooth. The complex reflection groups have been classified by Shephard and Todd
\cite{ShephardTodd}, the symplectic reflection groups by Cohen \cite{Cohen}.
The list of Shephard and Todd contains as a sublist the finite Coxeter groups.

The question which of these groups $G\subset \LieSp(V)$ admits a symplectic
resolution for $V/G$ has been solved for the Coxeter groups by Ginzburg and Kaledin
\cite{GinzburgKaledin} and for arbitrary complex reflection groups most
recently by Bellamy \cite{Bellamy}. His result is as follows:

\begin{theorem}\label{thm:A} {\em (Bellamy)} ---
If $G\subset \LieGl(V_0)$ is a finite complex reflection group, then $V_0\oplus V_0^*/G$
admits a symplectic resolution if and only if $(G,V_0)$ belongs to the following cases:
\begin{enumerate}
\item $(S_n,\gothh)$, where the symmetric group $S_n$ acts by permutations on the hyperplane $\gothh=\{x\in\IC^n\;|\;\sum_ix_i=0\}$.
\item $((\IZ/m)^n\rtimes S_n,\IC^n)$, the action being given by multiplication with $m$-th roots
of unity and permutations of the coordinates.
\item $(T, S_1)$, where $S_1$ denotes a two-dimensional representation of the binary tetrahedral group $T$
(see below).
\end{enumerate}
\end{theorem}

However, the technique of Ginzburg, Kaledin and Bellamy does not provide resolutions
beyond the statement of existence. Case 1 corresponds to Coxeter groups of type $A$
and Case 2 with $m=2$ to Coxeter groups of type $B$. It is well-known that symplectic
resolutions of $\gothh\oplus \gothh^*/S_n$ and $\IC^n\oplus\IC^n/(\IZ/m)^n\rtimes S_n\isom \Sym^n(\IC^2/(\IZ/m))$ are given as follows:

For a smooth surface $Y$ the Hilbert scheme $\Hilb^n(Y)$ of generalised $n$-tuples of points
on $Y$ provides a crepant resolution $\Hilb^n(Y)\to \Sym^n(Y)$. Applied to a minimal resolution of
the $A_{m-1}$-singularity $\IC^2/G$, $G\isom\IZ/m$, this construction yields a small resolution $\Hilb^n(\widetilde{\IC^2/G})\to \Sym^n(\widetilde{\IC^2/G})\to \Sym^n(\IC^2/G)$. Similarly,
$(\gothh\oplus \gothh^*)/S_n$ is the fibre over the origin of the barycentric
map $\Sym^n(\IC^2)\to \IC^2$. Thus $(\gothh\oplus \gothh^*)/S_n$ is resolved symplectically
by the null-fibre of the morphism $\Hilb^n(\IC^2)\to\Sym^n(\IC^2)\to \IC^2$.

It is the purpose of this note to describe an explicit symplectic resolution for the binary
tetrahedral group.

\section{The binary tetrahedral group}

Let $T_0\subset \LieSO(3)$ denote the symmetry group of a regular tetrahedron.
The preimage of $T_0$ under the standard homomorphism $\LieSU(2)\to \LieSO(3)$ is the
binary tetrahedral group $T$. As an abstract group, $T$ is the semidirect product of
the quaternion group $Q_8=\{\pm 1,\pm I,\pm J,\pm K\}$ and the cyclic group $\IZ/3$.
As a subgroup of $\LieSU(2)$ it is generated by the elements
$$I=\left(\begin{matrix}i&0\\0&-i
    \end{matrix}\right)
\quad\text{ and }\quad
\tau=-\frac12\left(\begin{matrix}
1+i&-1+i\\1+i&1-i
\end{matrix}\right)
$$
The binary tetrahedral group has 7 irreducible complex representations: A three-dimensional
one arising from the quotient $T\to T_0\subset\LieSO_3$, three one-dimensional representations
$\IC_j$ arising from the quotient $T\to \IZ/3$ with $\tau$ acting by $e^{2\pi ij/3}$, and
three two-dimensional representations $S_0$, $S_1$ and $S_2$. Here $S_0$ denotes the
standard representation of $T$ arising from the embedding $T\subset \LieSU_2$. This representation
is symplectic, its quotient $S_0/T$ being the well-known Klein-DuVal singularity of type $E_6$.
The two other representations can be written as $S_j=S_0\tensor \IC_j$, $j=1,2$. They are dual
to each other. It is as the subgroup of $\subset\LieGl(S_1)$ that $T$ appears in the list of Shephard
and Todd under the label ``No.~4''. The diagonal action of $T$ on $S_1\oplus S_2$ provides the
embedding of $T$ to $\LieSp_4$ that is of interest in our context.

Whereas the action of $T$ on $S_0$ is symplectic, the action of $T$ on $S_1$ and $S_2$ is generated
by complex reflections of order 3. Overall, there are 8 elements of order 3 in $T$ or rather
4 pairs of inverse elements, forming 2 conjugacy classes. To these correspond 4 lines in $S_1$ of
points with nontrivial isotropy groups. Let $C_1\subset S_1$ and $C_2\subset S_2$ denote the union
of these lines in each case. Then $C_1\times S_2$ and $S_1\times C_2$ are invariant divisors in
$S_1\oplus S_2$. However, the defining equations
are invariant only up to a scalar. Consequently, their images $W_1$ and $W_2$ in the
quotient $Z=S_1\oplus S_2/T$ are Weil divisors but not Cartier. The reduced singular locus
$\sing(Z)$ is irreducible and off the origin a transversal $A_2$ singularity. It forms one component
of the intersection $W_1\cap W_2$.

For $j=1,2$, let $\alpha_j:Z'_j\to Z$ denote the blow-up along $W_j$.
Next, let $W_j'$ be the reduced singular of locus $Z_j'$, and let $\beta_j:Z_j''\to Z_j'$ denote
the blow-up along $W_j'$.

\begin{theorem}--- The morphisms $\sigma_j=\alpha_j\beta_j:Z_j''\to Z$, $j=1,2$, are
symplectic resolutions.
\end{theorem}

\begin{proof} As all data are explicit, the assertion can be checked by brute calculation.
To cope with the computational complexity we use the free computer algebra system
SINGULAR\footnote{A documented SINGULAR file containing all the calculations is available from the authors upon request.} \cite{Greuel}. It suffices
to treat one of the two cases of the theorem. We indicate the basic steps for $j=2$.
In order to improve the readability of the formulae we write $q=\sqrt{-3}$.

Let $\IC[x_1,x_2,x_3,x_4]$ denote the ring of polynomial functions on $S_1\oplus S_2$.
The invariant subring $\IC[x_1,x_2,x_3,x_4]^T$ is generated by eight elements, listed in
table 1. The kernel $I$ of the corresponding ring homomorphism
$$\IC[z_1,\ldots,z_8]\to \IC[x_1,x_2,x_3,x_4]^T$$
is generated by nine elements, listed in table 2. The curve $C_2$ is given by the
semiinvariant $x_3^4+2\q x_3^2x_4^2+x_4^4$. In order to keep the calculation as simple
as possible, the following observation is crucial: Modulo $I$, the Weil divisor $W_2$
can be described by 6 equations, listed in table 3. This leads to a comparatively 'small'
embedding $Z_2'\to \IP^5_Z$ of $Z$-varieties.
Off the origin, the effect of blowing-up of $W_2$ is easy to understand even without any
calculation: the action of the quaternion normal subgroup $Q_8\subset T$ on $S_1\oplus
S_2\setminus\{0\}$ is free. The action of $\IZ/3=T/Q_8$ on \hbox{$S_1\oplus S_2/Q_8$} produces
transversal $A_2$-singularities along a smooth two-dimensional subvariety. Blowing-up along
$W_1$ or $W_2$ is a partial resolution: it introduces a $\IP^1$ fibre over each singular
point, and the total space contains a transversal $A_1$-singularity.

\begin{table}
Table 1: generators for the invariant subring $\IC[x_1,x_2,x_3,x_4]^T$:
 $$\begin{array}{ll}
z_1=x_1x_3+x_2x_4,&z_4=x_2x_3^3-\q x_1x_3^2x_4+\q x_2x_3x_4^2-x_1x_4^3,\\
z_2=x_3^4-2\q x_3^2x_4^2+x_4^4,&z_5=x_2^3x_3-\q x_1^2x_2x_3+\q x_1x_2^2x_4-x_1^3x_4,\\
z_3=x_1^4+2\q x_1^2x_2^2+x_2^4,&z_6=x_1^5x_2-x_1x_2^5,\\
z_7=x_3^5x_4-x_3x_4^5&z_8=x_1x_2^2x_3^3-x_2^3x_3^2x_4-x_1^3x_3x_4^2+x_1^2x_2x_4^3.\\
\end{array}$$
\end{table}

\begin{table}
Table 2: generators for $I=\ker(\IC[z_1,\ldots,z_8]\to \IC[x_1,\ldots,x_4]^T)$.
 $$
\begin{array}{ll}
\q z_1^3z_5-z_1z_3z_4-2z_2z_6-z_5z_8,&\quad z_1z_5^2+2z_4z_6+z_3z_8,\\
\q z_1^3z_4+z_1z_2z_5-2z_3z_7-z_4z_8,&\quad z_1z_4^2-2z_5z_7-z_2z_8,\\
-z_1^4+z_2z_3-z_4z_5-3\q z_1z_8,&\quad  \q z_1^2z_3z_5-2z_1^3z_6-z_3^2z_4+z_5^3-6\q z_6z_8,\\
z_1^2z_4z_5+\q z_1^3z_8+4z_6z_7-z_8^2,&\quad \q z_1^2z_2z_4-2z_1^3z_7-z_4^3+z_2^2z_5-6\q z_7z_8,\\
\multicolumn{2}{l}{
4z_1^2z_4z_5+\q z_3z_4^2-\q z_2z_5^2+4z_6z_7+8z_8^2}
\end{array}
$$
\end{table}

\begin{table}
Table 3: generators for the ideal of the Weil divisor $W_2\subset Z$.
$$\begin{array}{lll}
b_1=z_3z_7+2z_4z_8,&
b_2=z_2z_4+2\q z_1z_7,&
b_3=z_2z_3-4\q z_1z_8,\\
b_4=z_2^3+12\q z_7^2,&
b_5=z_1z_2^2-6z_4z_7,&
b_6=z_1^2z_2-\q z_4^2.
\end{array}$$
\end{table}

The homogeneous ideal $I'_2\subset \IC[z_1,\ldots,z_8,b_1,\ldots,b_6]$ that describes the
subvariety $Z_2'\subset \IP^5_Z$ is generated by $I$ and $39$ additional polynomials. In
order to understand the nature of the singularities of $Z_2'$ we consider the six affine
charts $U_\ell=\{b_\ell=1\}$. The result can be summarised like this: The singular locus
of $Z_2'$ is completely contained in $U_2\cup U_3$, so only these charts are relevant for
the discussion of the second blow-up. In fact, the corresponding affine coordinate rings
have the following description:
$$R_2=\IC[z_1,b_3,b_4,b_5,b_6]/ (b_5b_6-2\q z_1)^2+b_4(3\q b_3-b_6^3)$$
is a transversal $A_1$-singularity.
$$R_3=\IC[z_1,z_3,z_5,z_6,b_1,b_2,b_6]/J,$$ where $J$ is generated by five elements,
listed in table 4. Inspection of these generators shows that $\Spec(R_2)$ is isomorphic
to the singularity $(\gothh_3\oplus\gothh_3^*)/S_3$, the symplectic singularity of
Coxeter type $A_2$ that appears as case 1 in Bellamy's theorem. It is well-known that
blowing up the singular locus yields a small resolution. For arbitrary $n$, this is
a theorem of Haiman \cite[Prop.\ 2.6]{Haiman}, in our case it is easier to do
it directly. Thus blowing-up the reduced singular locus of $Z_2'$ produces a smooth
resolution $Z_2''\to Z$.

\begin{table}
Table 4: generators for the ideal sheaf $J$ of $Z_2'\subset \IC^7$ in the third chart:
$$\begin{array}{lll}
4z_1b_1+\q z_3b_2+z_5b_6,&
z_1z_5+z_3b_1+\q z_6b_6,\\
z_1^2b_6-z_3b_6^2-4\q b_1^2-3 z_5b_2,&
z_1^2z_3-z_3^2b_6-\q z_5^2-12z_6b_1,\\
z_1^3-z_1z_3b_6+\q z_5b_1+3\q z_6b_2
\end{array}$$
\end{table}

It remains to check that the morphism $\alpha_2:Z_2\to Z$ is semi-small. For this it
suffices to verify that the fibre $E=(\alpha_2^{-1}(0))_{\red}$ over the origin is
two-dimensional and not contained in the singular locus of $Z_2'$. Indeed, the computer
calculation shows that $E\subset \IP^5$ is given by the equations $b_1$,  $b_3b_5$, $b_3b_4$,
$b_5^2-b_4b_6$ and hence is the union of two irreducible
surfaces. The singular locus of $Z_2'$ is irreducible and two-dimensional and dominates
the singular locus of $Z$. Thus the second requirement is fulfilled, too.
\end{proof}

\section{The equivariant Hilbert scheme}

Though the description of the resolution is simple and straight the method of proof is less
satisfying. It is based on explicit calculation that given the complexity of the singularity
we were able to carry out only by means of appropriate software.
Remark that even for the classical ADE-singularities arising from finite subgroups
$G\subset\LieSU(2)$ the actual resolutions of $\IC^2/G$ could only be described by explicit
calculations. The difference to our case essentially is one of complexity: The dimension is
four instead of two, there are 8 basic invariants satisfying 9 relations instead of Klein's
three invariants with a single relation, and the singular locus is itself a complicated singular variety instead
of an isolated point. The first construction that resolved the ADE-singularities in a uniform way was
given by Kronheimer in \cite{Kronheimer} in terms of certain hyper-K\"ahler quotients. Later
Ito and Nakamura \cite{ItoNakamura} used $G$-Hilbert schemes to the same effect.

Recall that given a scheme $X$ with an action of a finite group $G$ the equivariant Hilbert
scheme $G$-$\Hilb(X)$ is the moduli scheme of zero-dimensional equivariant subschemes
$\xi\subset X$ such that $H^0(\xi,\ko_\xi)$ is isomorphic to the regular representation of $G$.
There is a canonical morphism $G$-$\Hilb(X)\to X/G$ which is an isomorphism over the open subset
that corresponds to regular orbits. For a linear action $G\subset \LieSl(V)$ it is known that
$G$-$\Hilb(V)$ is smooth if $\dim(V)=2$ or $3$ and provides a crepant resolution of $V/G$
(see \cite{BKR}).

Thus $\rho: H:=T$-$\Hilb(S_1\oplus S_2)\to Z=(S_1\oplus S_2)/T$ is a natural candidate for a resolution.
As the generic singularity of $Z$ is a transversal $A_2$ singularity that arises
from a $\IZ/3$-action it is clear that $\rho$ is a crepant resolution off the origin. However, it
turns out that $H$ has two irreducible components that are smooth and intersect transversely.
One of them is the closure $H^{\orb}$ of the locus of regular orbits, it dominates the quotient $Z$.
This orbit component also appears as 'dynamical component' or 'coherent component' in the literature.
Any other component of $H$ must be contained in the fibre $\rho^{-1}(0)$, though this is not true in
general.

The two factors of the group $\IC^*\times\IC^*$ act on $S_1\oplus S_2$ via dilations on the
first and second summand, respectively, and the polynomial ring $\IC[S_1\oplus S_2]$ may
accordingly be decomposed into irreducible $T\times \IC^*\times \IC^*$--representations, the first
terms being
\begin{eqnarray*}\IC[S_1\oplus S_2]&=&L_0\oplus (S_1x \oplus S_2y)\oplus R_0x^2\oplus (R_0\oplus L_0)xy \oplus R_0y^2\\
&& \oplus (S_1\oplus S_2)(x^3\oplus y^3) \oplus (S_0\oplus S_1\oplus S_2)(x^2y\oplus xy^2)\oplus \ldots
\end{eqnarray*}
where $x$ and $y$ are formal symbols indicating the weight with respect to the $\IC^*\times \IC^*$ action.
Using this decomposition one can see that $H$ contains a further component isomorphic to
$\IP^2\times\IP^2$: if $I\in H$ is to be an ideal contained in the square of the maximal ideal generated
by $S_1x\oplus S_2y$, of the respectively three copies of $S_1$ and $S_2$ of total weight 3 two have
to be contained in $I$. The possible choices amount to picking a line in a three-dimensional space for each
of $S_1$ and $S_2$. Of course, one still needs to check that every choice really leads to an admissable ideal.

As the map $\rho$ is proper, each equivariant closed subset of $H$ must contain fixed points for the
$\IC^*\times \IC^*$-action. These correspond to $T$-equivariant bihomogeneous ideals $I\subset \IC[S_1\oplus S_2]$.
Using the given decomposition of the coordinate ring it is not difficult to see
that there are 13 such fixed points $I_i\in H$. The tangent space to $H$ at $I_i$ ist given by
$\Hom_T(I_i, \IC[S_1\oplus S_2]/I_i)$. An explicit calculation shows that the dimension of the tangent
space is 4 in seven points (necessarily smooth points of $H$) and is 5 in six other points. The calculation of
the quadratic component of the analytic obstruction or Kuranishi map $\Hom_T(I_i, \IC[S_1\oplus S_2]/I_i)\to
\Ext^1_T(I_i, \IC[S_1\oplus S_2]/I_i)$ yields in all cases a reducible quadric with two distinct factors. This
suffices to conclude that there are no further components of $H$, that $H^{\orb}$ is smooth and that the two
components  $H^{\orb}$ and $\IP^2\times \IP^2$ intersect transversely. By the universal property of the blow-up
there is a commuting diagram
$$\begin{array}{ccccc}
&&H^{\orb}\\
&\swarrow&&\searrow\\
Z_1''&&&&Z_2''\\
&\searrow&&\swarrow\\
&&Z
  \end{array},
$$
which conjecturely relates the two-resolutions by a Mukai-flop. In fact, we found the two resolutions first by contracting
local models of $H$ that are given as subschemes of a relative Grassmannian over $Z$.
The calculations so far described are insufficient to formally prove that $H^{\orb}$ and $\IP^2\times\IP^2$ intersect along the incidence
variety and that this intersection is  the exceptional locus for the two contractions.
However, recall that equivariant Hilbert schemes can be seen as special cases of quiver varieties
(\cite{Nakajima}, \cite{King}) for the McKay quiver \cite{McKay} associated to the given action. As the referee
suggests one might try to obtain the diagram above and resolutions of $Z$ in a single stroke by a variation of the
stability condition in the construction of the quiver variety. As the McKay quiver for the action of $T$ on $S_0$
is the Dynkin graph of type $E_6$, it is easy to deduce the graph underlying the McMay quiver for the action of $T$
on $S_1\oplus S_2$:
\begin{center}
\unitlength=.8mm
\begin{picture}(50,50)
\put(20,20){
\put(0,0){$R$}
\put(24,0){$L_0$}
\put(-24,0){$S_0$}
\put(12,20){$S_1$}
\put(12,-20){$S_2$}
\put(-12,20){$L_2$}
\put(-12,-20){$L_1$}
\put(-5,22){\line(1,0){15}}
\put(-5,-18){\line(1,0){15}}
\put(-2,2.5){\line(-1,0){15}}
\put(-2,1.5){\line(-1,0){15}}
\put(5,6){\line(1,2){5}}
\put(4,6){\line(1,2){5}}
\put(5,-4){\line(1,-2){5}}
\put(4,-4){\line(1,-2){5}}
\put(-19,6){\line(1,2){5}}
\put(-19,-4){\line(1,-2){5}}
\put(24,6){\line(-1,2){5}}
\put(24,-4){\line(-1,-2){5}}
}
\end{picture}
\end{center}

\bibliographystyle{plain}

\parindent0mm

\end{document}